\documentclass[11pt,reqno,tbtags,a4paper]{amsart}
\usepackage{amssymb}
\usepackage{mathabx}  
\usepackage{xpunctuate}
\usepackage{url}
\usepackage[square,numbers]{natbib}
\bibpunct[, ]{[}{]}{;}{n}{,}{,}

\title[Back and forward arcs in Depth First Search in a random digraph]
{On Knuth's conjecture 
for back and forward arcs in Depth First Search in a random digraph with
geometric outdegree distribution}

\date{10 January, 2023}

\author{Svante Janson}
\thanks{Supported by the Knut and Alice Wallenberg Foundation}
\address{Department of Mathematics, Uppsala University, PO Box 480,
SE-751~06 Uppsala, Sweden}
\email{svante.janson@math.uu.se}
\newcommand\urladdrx[1]{{\urladdr{\def~{{\tiny$\sim$}}#1}}}
\urladdrx{http://www2.math.uu.se/~svante/}

\subjclass[2020]{} 

\numberwithin{equation}{section}

\renewcommand\le{\leqslant}
\renewcommand\ge{\geqslant}

\allowdisplaybreaks

\theoremstyle{plain}
\newtheorem{theorem}{Theorem}[section]

\newtheorem{proposition}[theorem]{Proposition}
\newtheorem{corollary}[theorem]{Corollary}
\newtheorem{conj}[theorem]{Conjecture}

\theoremstyle{definition}

\newtheorem{exampleqqq}[theorem]{Example}

\newtheorem{remarkqqq}[theorem]{Remark}
\newenvironment{remark}{\begin{remarkqqq}}
  {\hfill\qedsymbol\end{remarkqqq}}

\newtheorem{problem}[theorem]{Problem}

\theoremstyle{remark}

\newenvironment{romenumerate}[1][-10pt]{
\addtolength{\leftmargini}{#1}\begin{enumerate}
 \renewcommand{\labelenumi}{\textup{(\roman{enumi})}}%
 }{\end{enumerate}}

\newcounter{oldenumi}
{\setcounter{oldenumi}{\value{enumi}}
\begin{romenumerate} \setcounter{enumi}{\value{oldenumi}}}
{\end{romenumerate}}

\newcounter{thmenumerate}

\newcounter{xenumerate}   

\newcommand{\refC}[1]{Corollary~\ref{#1}}

\newcommand{\refS}[1]{Section~\ref{#1}}

\newcommand\marginal[1]{\marginpar[\raggedleft\tiny #1]{\raggedright\tiny#1}}
\newcommand\REM[1]{{\raggedright\texttt{[#1]}\par\marginal{XXX}}}
\newcommand\XREM[1]{\relax}

\begingroup
  \divide\count255 by 60
  \multiply\count255 by -60
  \advance\count255 by \time
  \ifnum \count255 < 10 \xdef\klockan{\the\count1.0\the\count255}
  \else\xdef\klockan{\the\count1.\the\count255}\fi
\endgroup

\newcommand{\sumio}{\sum_{i=0}^\infty}

\newcommand\bigpar[1]{\bigl(#1\bigr)}
\newcommand\Bigpar[1]{\Bigl(#1\Bigr)}

\newcommand\bigsqpar[1]{\bigl[#1\bigr]}

\def\rompar(#1){\textup(#1\textup)}    

\def\xexp(#1){e^{#1}}

\newcommand\punkt{\xperiod}    
    
\newcommand\ie{i.e\punkt}

\newcommand\eqd{\overset{\mathrm{d}}{=}}

\newcounter{CC}
\newcounter{cc}

\newcommand\E{\operatorname{\mathbb E{}}}
\renewcommand\P{\operatorname{\mathbb P{}}}

\newcommand\Ge{\operatorname{Ge}}

\renewcommand\phi{\xxx}  

\newcommand\qw{^{-1}}

\newcommand\GX{\widehat G}
\newcommand\GY{\widecheck G}

\begin{document}

%
%
%
%

\dedicatory{
To Donald Knuth on his 85th birthday}

\begin{abstract}
Donald Knuth, 
in a draft of a coming volume of \emph{The Art of Computer Programming},
has recently conjectured 
that in Depth-First Search of a random digraph with geometric outdegree
distribution,
the numbers of back and forward arcs have the same distribution.

We show that this conjecture is equivalent to an equality
between two generating functions defined by different recursions.

Unfortunately, we have not been able so use this to prove the conjecture,
which still is open,
but we hope that this note will inspire others to succeed 
with the conjecture.
\end{abstract}

\maketitle


\section{Introduction}\label{S:intro}

Donald Knuth \cite[Section 7.4.1.2]{Knuth12A}, 
in a draft of a coming volume of \emph{The Art of Computer Programming},
studies  Depth-First Search (DFS) in a digraph.
Among the results there, he makes an
intriguing conjecture 
\cite[Problem 7.4.1.2-35]{Knuth12A} 
for the case when the digraph is random with a geometric
outdegree distribution; see \refS{Sconj} below.
The conjecture is strongly supported by exact calculations showing that it
holds for small digraphs 
(order $n\le9$, extend by us by the methods below to $n\le18$).

We have not been able to prove this conjecture, but we show below 
(\refC{CKnuth})
that the
conjecture is equivalent to the equality of two generating functions defined
by similar but different recursions. 
(We also give an extended version of the conjecture.)
The purpose of this note is to present this reformulation in the hope of
stimulating others to show the conjecture.

\section{Knuth's conjecture}\label{Sconj}
\subsection{Depth-First Search}
For a detailed description of DFS in a digraph, 
and for historical notes,
see \cite[Section 7.4.1.2]{Knuth12A}.
Briefly, the DFS starts with an arbitrary vertex, and explores the arcs from
that vertex one by one. When an arc is found leading to a  vertex that
has not been seen before, the DFS explores the arcs from it in the same way,
in a recursive fashion, before returning to the next arc from its parent.
This eventually yields a tree containing all descendants of the the first
vertex.
If there still are some unseen vertices, the DFS starts again with one of
them and finds a new tree, and so on until all vertices are found.
The DFS thus yields a spanning forest in the digraph, called the 
\emph{depth-first forest}.

The discarded arcs (leading to already seen vertices) may be classified
further; 
\cite[Section 7.4.1.2]{Knuth12A} 
classifies the arcs in the
digraph into the following five types
\begin{itemize}
    \item \emph{loops};
    \item \emph{tree arcs}, the arcs in the resulting depth-first forest;
    \item \emph{back arcs}, 
the arcs that point to an ancestor of the current vertex in the current tree;
    \item \emph{forward arcs}, 
the arcs that point to an already discovered descendant of the current
vertex in the current tree;
    \item \emph{cross arcs}, 
all other arcs (these point to an already discovered vertex which is neither
a descendant nor an ancestor of the current vertex, and might be in another
tree).
\end{itemize}

Each instance of DFS in a digraph is described by the sequence of arcs
that are explored, taken in the order they are explored.
For an instance of the DFS,
let 
$L,F,B,C,T$ be the numbers of loops,
forward, backward, cross, and tree arcs.

Let $F_n(w,x,y,z,t)$ be the generating function defined as the sum of
$w^Lx^Fy^Bz^Ct^T$ over all instances of a DFS of a digraph on $n$ vertices
labelled $1,\dots,n$, where we for definiteness 
assume that when the DFS starts a new tree, it chooses as the root
the unused vertex with smallest label.
(Algorithm D in \cite[7.4.1.2]{Knuth12A} chooses the unused  vertex with largest
label;
this is obviously equivalent for our purposes, 
but notationally less convenient for us.)

\subsection{Random digraphs with geometric outdegree distribution} 
From now on we consider DFS in a random digraph given by the following
model.

The number of vertices $n$ is given; we are also given a parameter $p\in(0,1)$.
Each vertex is given an outdegree that is a random variable with
a geometric distribution $\Ge(1-p)$;
in other words,
if the outdegree of vertex $i$ is $\eta_i$, then
\begin{align}
  \P(\eta_i=k) = p^k(1-p),
\qquad k=0,1,2,\dots;
\end{align}
moreover,
the outdegrees of different vertices are independent.
This gives each vertex $i$ a number $\eta_i$ of outgoing arcs; the other 
endpoint of these arcs are chosen uniformly at random among all $n$
vertices, independently for all arcs. (Thus, loops and multiple arcs are
allowed, so the digraph is really a multidigraph.)

\subsection{The conjecture}
Donald Knuth 
\cite[Problem 7.4.1.2-35]{Knuth12A} 
conjectures
that for a digraph with a geometric outdegree distribution,
for every $n\ge1$ and every $p\in(0,1)$,
the random variables $F$ and $B$ have the same distribution, which we write
as
\begin{align}\label{kn1}
F\eqd B.  
\end{align}
Equivalently, in terms of probability generating functions,  
\begin{align}\label{kn2}
\E x^F=\E x^B.   
\end{align}
(Knuth verified this for $n\le 9$ by explicit calculations.)

\subsection{The conjecture and generating functions}
Given a sequence of arcs that describes an instance of a DFS, we find the
probability of observing exactly this sequence as follows:
for each arc in the sequence, the emitting vertex should send out at least
one more arc after what we may have seen earlier (probability $p$), and the
endpoint of this arc should be a specified vertex (probability $1/n$);
furthermore, for each vertex, there is one time when this vertex had a chance
to send out more vertices but did not do so (probability $1-p$).
Hence, the probability is $(p/n)^{L+F+B+C+T}(1-p)^n$.
Consequently, the joint probability generating function of $(L,F,B,C,T)$
for the DFS in a random digraph with $n$ vertices and
$\eta\sim\Ge(1-p)$ is
\begin{align}\label{a1}
\E \bigsqpar{w^Lx^Fy^Bz^Ct^T}
=
  (1-p)^nF_n
\Bigpar{w\frac{p}{n},x\frac{p}{n},y\frac{p}{n},z\frac{p}{n},t\frac{p}{n}}.
\end{align}

By \eqref{a1}, Knuth's conjecture \eqref{kn1}
is thus equivalent to
\begin{align}\label{a2}
F_n \Bigpar{\frac{p}{n},x\frac{p}{n},\frac{p}{n},\frac{p}{n},\frac{p}{n}}
=
F_n\Bigpar{\frac{p}{n},\frac{p}{n},x\frac{p}{n},\frac{p}{n},\frac{p}{n}},
\end{align}
or, by first replacing $p/n$ by $w$ and then $xw$ by $x$, 
\begin{align}\label{a3}
F_n \bigpar{w,x,w,w,w}
=
F_n\bigpar{w,w,x,w,w}
.
\end{align}

In fact, we conjecture, more generally, that
for every $n\ge1$,
\begin{align}\label{a4}
F_n \bigpar{w,x,z,z,t}
=
F_n\bigpar{w,z,x,z,t},
\end{align}
which by the same argument is equivalent to
\begin{align}\label{a5}
  \E \bigsqpar{w^L x^F z^{B+C} t^F}
=
\E \bigsqpar{w^L x^B z^{F+C} t^F},
\end{align}
and  thus to the following extended version of Knuth's conjecture, with joint
distributions of different variables.
\begin{conj}
For every $n\ge1$ and every $p\in(0,1)$,
  \begin{align}\label{a6}
  (L,F,B+C,T)\eqd(L,B,F+C,T).
\end{align}
\end{conj}

\begin{remark}
  The much weaker statement that $\E F=\E B$, i.e., that the expectations
  are the same, is proved in \cite[Theorem 2.18]{SJ364}.
It is there also shown that $\E T = \E C$, but that $T$ and $C$ do not have
the same distribution.
\end{remark}

\begin{remark}
  It is easy to see that \eqref{a6} cannot be extended to 
$  (L,F,B,C,T)\eqd(L,B,F,C,T)$.
This fails already for $n=3$, as can be seen by the argument in \refS{Srec},
calculating $G_3(w,x,y,z)$ by \eqref{b3} and noting that 
$G_3(w,x,y,z)\neq G_3(w,y,x,z)$.
\end{remark}

\section{Recursion formulas}\label{Srec}

To find $F_n$, consider now the DFS stopped when the first tree is
completed,
and define $G_m(w,x,y,z,t)$ as the sum of $w^Lx^Fy^Bz^Ct^T$ over all
sequences of arcs that yield the first tree in an instance of DFS,
where moreover this tree has $m$ given vertices and these are visited in a
given order 
(for example, the vertices $1,2,\dots,m$);
we here define $L,F,B,C,T$ as above also for sequences of arcs that
are only a part of the full DFS.

In general, a DFS creates a forest with several trees. 
Consider the contribution to $F_n$ from the case when 
the depth-first forest have $k$ given trees with sizes $m_1,m_2,\dots,m_k$,
with vertices visited in a given order.
The exploration of the first tree gives a factor $G_{m_1}(w,x,y,z,t)$.
The exploration of the second tree proceeds like an exploration of just that
tree, but we may also have cross arcs that go back to the any of the $m_1$
vertices of the first tree; in fact, any time that we may have a loop we may
instead have one of these cross arcs, and therefore the exploration of the
second tree yields a factor $G_{m_2}(w+m_1z,x,y,z,t)$, where $w$ is replaced
by $w+m_1z$ to account for the possible additional cross arcs.
The same reasoning applies to all following trees, and it follows that the
contribution to $F_n(w,x,y,z,t)$ 
is $\prod_{i=1}^k G_{m_i}(w+M_{i-1}z,x,y,z,t)$, 
where $M_j:=\sum_{i\le j}m_i$,
and thus $F_n(w,x,y,z,t)$ equals the
sum over all possible depth-first forests of such products.
Consequently, \eqref{a4} follows if, for every $m\ge1$,
\begin{align}\label{a7}
G_m \bigpar{w,x,z,z,t}
=
G_m\bigpar{w,z,x,z,t}.
\end{align}
Conversely, it follows by induction that \eqref{a7} is necessary for
\eqref{a4} to hold for all $n$.

In the definition of $G_m$, we consider only sequences of arcs that produce
a tree with $m$ vertices; thus the number of tree edges $T=m-1$.
Hence we have
\begin{align}\label{a8}
  G_m(w,x,y,z,t)=t^{m-1}G_m(w,x,y,z,1).
\end{align}
We simplify the notation by writing $G_m(w,x,y,z):=G_m(w,x,y,z,1)$, and see that
\eqref{a7} is equivalent to
\begin{align}\label{a9}
G_m \bigpar{w,x,z,z}
=
G_m\bigpar{w,z,x,z}.
\end{align}

We find a recursion formula for $G_n$ (now replacing $m$ by $n$ for
convenience). 
For $n=1$, the sequences of arcs that appear in the definition of $G_1$ are
the sequences of 0 or more loops at the root. Hence, $F=B=C=0$ and
\begin{align}\label{a10}
  G_1(w,x,y,z)=\sumio w^i=\frac{1}{1-w}.
\end{align}

For $n\ge2$, suppose that the DFS produces a tree $\tau$ on the vertices
$1,\dots,n$ (in this order).
Consider the last child of the root, and let it have label $m+1$ where $1\le
m\le n-1$. Let $\tau'$ be the subtree of $\tau$ formed by the vertices
before $m+1$, \ie{} $1,\dots,m$, and let $\tau''$ be the subtree formed 
by $m+1$ and its descendants; these have sizes $|\tau'|=m$ and
$|\tau''|=n-m$.
An exploration of $\tau$ consists of 
\begin{enumerate}
\item an exploration of $\tau'$,
\item a tree arc from the root 1 to $m+1$,
\item an exploration of $\tau''$,
\item\label{taud} return to the root and a number of arcs leading only to already seen
  vertices. 
\end{enumerate}
In order for a sequence of arcs to produce the tree $\tau$, we have as long
as we are exploring inside $\tau'$
exactly the same possibilities as for the tree $\tau'$ on its own,
but during the exploration of $\tau''$, we can also have edges to vertices
in $\tau'$; more precisely, we see that for any exploration yielding the
tree $\tau''$, each time we may have a loop, we can in the exploration of $\tau$
also have a back edge to the root 1, or a cross edge to one of the $m-1$ other
vertices in $\tau'$. This corresponds to replacing $w$ by $w+y+(m-1)z$ in
the generating function.
Moreover, in step \ref{taud}, we can have any number $\ge0$ of arcs 
from 1 to itself (a loop) or to one of the $n-1$ other vertices in $\tau$ (a
forward arc). This yields for the generating function a factor
\begin{align}
  \label{b2}
\sumio (w+(n-1)x)^i=\frac{1}{1-w-(n-1)x}.
\end{align}
Consequently, summing over all possible $m$ and trees $\tau'$ and $\tau''$, 
we find, for $n\ge2$,
\begin{align}\label{b3}
G_n(w,x,y,z)
=\frac{1}{1-w-(n-1)x}\sum_{m=1}^{n-1}G_m(w,x,y,z)G_{n-m}(w+y+(m-1)z,x,y,z).
\end{align}
We have shown the following:
\begin{proposition}
  The generating function $G_n$ satisfies the recursion \eqref{b3} with the
  initial value \eqref{a10}.
\end{proposition}

Specializing to the parameters in \eqref{a9}, we have thus the following 
corollary.
Here
\begin{align}
  \label{G^}
\GX_n(w,x,z):=G_n(w,x,z,z)
\quad\text{and}\quad
\GY_n(w,x,z):=G_n(w,z,x,z)
\end{align}
are generating
functions for $(L,F,B+C)$ and $(L,B,F+C)$, respectively.
\begin{proposition}
  Let $\GX_n(w,x,z)$ and $\GY_n(w,x,z)$ be defined by the recursions
  \begin{align}\label{b4}
    \GX_1(w,x,z)&:=\GY_1(w,x,z):=\frac1{1-w}
  \end{align}
and, for $n\ge2$,
  \begin{align}\label{b5}
\GX_n(w,x,z)&:=
\frac{1}{1-w-(n-1)x}\sum_{m=1}^{n-1}\GX_m(w,x,z)\GX_{n-m}(w+mz,x,z),
\\\label{b6}
\GY_n(w,x,z)&:=
\frac{1}{1-w-(n-1)z}\sum_{m=1}^{n-1}\GY_m(w,x,z)\GY_{n-m}(w+x+(m-1)z,x,z).
  \end{align}
Then \eqref{a6}, or equivalently \eqref{a4}, holds if and only if
\begin{align}\label{b7}
  \GX_n(w,x,z)=\GY_n(w,x,z),
\qquad n\ge1.
\end{align}
\end{proposition}

\begin{proof}
First, $\GX_n$ and $\GY_n$ satisfy \eqref{b4}--\eqref{b6} by \eqref{G^} and
\eqref{b3},
so they can be defined by these recursions.

Furthermore, we have shown that
$\eqref{a6}\iff\eqref{a4}\iff\eqref{a7}\iff\eqref{a9}$,
and $\eqref{a9}\iff\eqref{b6}$ by \eqref{G^}.
\end{proof}

In order to show \eqref{a3}, \ie{} Knuth's conjecture, it suffices to 
verify \eqref{b7} for $z=w$. In other words:

\begin{corollary}\label{CKnuth}
  Knuth's conjecture $F\eqd B$, or equivalently \eqref{a3},
is equivalent to 
\begin{align}\label{b7=}
  \GX_n(w,x,w)=\GY_n(w,x,w),
\qquad n\ge1,
\end{align}
where $\GX_n$ and $\GY_n$ are given by the recursions \eqref{b4}--\eqref{b6}.
\end{corollary}

\begin{problem}
  Prove \eqref{b7}, or at least \eqref{b7=}!
\end{problem}

\begin{remark}
  Knuth verified his conjecture by exact calculations for $n\le 9$.
We have extended this by verifying \eqref{b7} (and thus \eqref{a4})
for $n\le 18$.
\end{remark}

\begin{remark}
  Another recursion for $\GY(w,x,z)$ is 
  \begin{align}\label{mag}
\GY_n(w,x,z)&:=
\frac{1}{1-w}\sum_{m=1}^{n-1}\GY_m(w+x,x,z)\GY_{n-m}(w+mz,x,z).
  \end{align}
This is obtained similarly as the recursion above, but now defining 
$\tau'$ to be the subtree of $\tau$ consisting of the first child of the
root (\ie{} 2) and all its descendants, and $\tau''$ to be the subtree
consisting of all other vertices (including the root 1).
An exploration of $\tau$ consists of 
\begin{enumerate}
\item\label{tauga} a number $\ge0$ of loops at the root 1,
\item a tree arc from 1 to 2,
\item\label{taugc} an exploration of $\tau'$,
\item\label{taugd} return to the root and an exploration of $\tau''$.
\end{enumerate}
Step \eqref{tauga} yields a factor $(1-w)\qw$ in the generating function.
Step \eqref{taugc} is as an exploration of the tree $\tau'$, but
each time we may have a loop, we may also have a back edge to 1;
this corresponds to replacing $w$ by $w+y$ in the generating function.
Similarly, in Step \eqref{taugd}, each time we may have a loop, we may also
have an edge to one of the vertices in $\tau'$; this is either a forward
edge (if we are at the root) or a cross edge; this corresponds to replacing
$w$ by $w+mz$ in the generating function, with $m:=|\tau'|$.
Together, this yields \eqref{mag}.

Note that although it seems possible to derive also a recursion for 
$\GX_n(w,x,z)$, or more generally for
$G_n(w,x,y,z)$, by the decomposition used here, it will be more complicated
because we then would have to keep control over whether edges from $\tau''$
to $\tau'$ are forward or cross edges. We leave this to the reader.
\end{remark}

\newcommand\AAP{\emph{Adv. Appl. Probab.} }
\newcommand\JAP{\emph{J. Appl. Probab.} }
\newcommand\JAMS{\emph{J. \AMS} }
\newcommand\MAMS{\emph{Memoirs \AMS} }
\newcommand\PAMS{\emph{Proc. \AMS} }
\newcommand\TAMS{\emph{Trans. \AMS} }
\newcommand\AnnMS{\emph{Ann. Math. Statist.} }
\newcommand\AnnPr{\emph{Ann. Probab.} }
\newcommand\CPC{\emph{Combin. Probab. Comput.} }
\newcommand\JMAA{\emph{J. Math. Anal. Appl.} }
\newcommand\RSA{\emph{Random Structures Algorithms} }
\newcommand\DMTCS{\jour{Discr. Math. Theor. Comput. Sci.} }

\newcommand\AMS{Amer. Math. Soc.}
\newcommand\Springer{Springer-Verlag}
\newcommand\Wiley{Wiley}

\newcommand\vol{\textbf}
\newcommand\jour{\emph}
\newcommand\book{\emph}
\newcommand\inbook{\emph}
\def\no#1#2,{\unskip#2, no. #1,} 
\newcommand\toappear{\unskip, to appear}

\newcommand\arxiv[1]{\texttt{arXiv}:#1}
\newcommand\arXiv{\arxiv}

\newcommand\xand{and }
\renewcommand\xand{\& }

\def\nobibitem#1\par{}

\end{document}